\documentclass{article}
\usepackage[T1]{fontenc}
\usepackage[utf8]{inputenc}
\usepackage[british]{babel}
\usepackage{amsmath,amssymb,amsthm}
\usepackage[all]{xy}
\usepackage[left = 3cm, right = 3cm]{geometry}

\newtheorem{theorem}{Theorem}

\theoremstyle{definition}

\theoremstyle{remark}

\title{On finite trifactorised groups and Sylow and Hall theorems for skew braces}

\author{A. Ballester-Bolinches%
\thanks{Departament de Matem\`atiques, Universitat de Val\`encia, Av.\ Vicent Andrés Estellés, 19, 46100 Burjassot, Val\`encia, Spain; \texttt{Adolfo.Ballester@uv.es}, \texttt{Pedro.A.Perez@uv.es}, \texttt{Vicent.Perez-Calabuig@uv.es}; ORCID 0000-0002-2051-9075, 0009-0009-7082-9002, 0000-0003-4101-8656}
\and P. P\'erez-Altarriba\addtocounter{footnote}{-1}\footnotemark \and V. P\'erez-Calabuig\addtocounter{footnote}{-1}\footnotemark}

\begin{document}

\maketitle

\begin{abstract}
The aim of this short note is to show that the Sylow theorem (respectively Hall theorem) for finite skew braces proved by Truman in \cite{truman} is a direct consequence of the Sylow structure (resp. Hall structure) of a finite trifactorised group. A Cauchy theorem for finite skew braces naturally emerges. 
\end{abstract}



\bigskip

\emph{Every group considered is finite.}

We say that a group $G$ is a \emph{trifactorised group} if there exist subgroups $C$, $D$ and $K$ such that  $G = KC = KD = DC$. These groups naturally appear in proofs about groups factorised by two subgroups. They are also important in the structural study of skew braces. 

Let $(B,+,\cdot)$ be a skew brace, and call $K = (B,+)$ and $C = (B,\cdot)$. Then $C$ acts on $K$  by means of the \emph{lambda} action. The corresponding semidirect  $G$ with respect to this action is a trifactorised group $G = KC = KD = DC$, where $D$ is the diagonal subgroup (cf. \cite{categ_trip}).

Recall that a group $G$ is said to \emph{satisfy the $\pi$-Sylow theorem} (also called $D_\pi$ property) for a set of primes $\pi$ if every $\pi$-subgroup is contained in a Hall $\pi$-subgroup and every pair of Hall $\pi$-subgroups are conjugated (e.g. every soluble group satisfies the $\pi$-Sylow theorem for every set of primes~$\pi$).

It is well-known that in a factorised group, there exist a Sylow $p$-subgroup of the group which is the product of Sylow $p$-subgroups of the factors for each prime $p$. More generally we have:

\begin{theorem}[{\cite[Kapitel~VI. Satz~4.6]{Huppert}}]
\label{teo:pi-Sylowtheroem}
Let $\pi$ be a set of primes and let $G = AB$ be the product of the subgroups $A$ and~$B$. Assume also that $G$, $A$ and $B$ satisfies the $\pi$-Sylow theorem. Then there exist Hall $\pi$-subgroups $G_\pi$, $A_\pi$, and $B_\pi$ of $G$, $A$, and $B$, respectively, such that $G_\pi = A_\pi B_\pi$ is a Hall $\pi$-subgroup of~$G$. 
\end{theorem}

As a consequence we have:

\begin{theorem}
\label{teo:main}
Let $\pi$ be a set of primes, and let $G = KC = KD = DC$ be a trifactorised group, where $K$ is normal in $G$. Assume that $G$, $D$ and $C$ satisfy the $\pi$-Sylow theorem. Then, there exists a Hall $\pi$-subgroup of $G$ which is a trifactorised subgroup $G_\pi = K_\pi C_\pi= K_\pi D_\pi = D_\pi C_\pi $, where $K_\pi$, $D_\pi$ and $C_\pi$ are Hall $\pi$-subgroups of $K$, $D$ and~$C$, respectively.
\end{theorem}

\begin{proof}
By Theorem~\ref{teo:pi-Sylowtheroem}, there exist $C_\pi$ and $D_\pi$ Hall $\pi$-subgroups of $C$ and $D$, respectively, such that $G_\pi = D_\pi C_\pi$ is a Hall $\pi$-subgroup of $G$. Consider $K_\pi = G_\pi \cap K$. Since $K \unlhd G$, it follows that $K_\pi$ is a Hall $\pi$-subgroup of~$K$. Hence, $G_\pi = D_\pi C_\pi = K_\pi D_\pi = K_\pi C_\pi$ is a Hall trifactorised $\pi$-subgroup of~$G$.
\end{proof}

Let $(B,+,\cdot)$ be a skew brace, call $K = (B,+)$ and $C = (B,\cdot)$, and let $G = KC = KD = DC$ be the associated trifactorised group. If $G$ satisfies the hypotheses of Theorem~\ref{teo:main}, then $C_\pi$ acts on $K_\pi = G_\pi \cap K$ by the restriction of the $\lambda$-map. Thus, $D_\pi$ is the associated diagonal subgroup to $K_\pi$ and~$C_\pi$. Therefore, $(K_\pi,+,\cdot)$ is a subbrace of $(B,+,\cdot)$. Consequently, a Sylow theorem (\cite[Theorem~2.1]{truman}) and a Hall theorem (\cite[Theorem~2.3]{truman}) for skew braces are corollaries of Theorem~\ref{teo:main}. Furthermore, according to~\cite[Theorem~A]{soluble} every skew brace without subbraces is a trivial skew brace isomorphic to a cyclic group of prime order. Hence, a Cauchy theorem (\cite[Theorem~2.2]{truman}) also holds for skew braces.

\section*{Acknowledgements}
\begin{sloppypar}
This work is supported by the grant PID2024-159495NB-I00, funded by MICIU/AEI/10.13039/501100011033 and by ERDF/EU, and by the grant CIAICO/2023/007 from the Conselleria d’Educació, Cultura, Universitats i Ocupació of the Generalitat Valenciana. The second author is funded by the Spanish Ministry of Science, Innovation and Universities through an FPU predoctoral grant FPU24/01319.
\end{sloppypar}

\bibliographystyle{plain}

\begin{thebibliography}{99}
\bibitem{soluble}
A.~Ballester-Bolinches, R.~Esteban-Romero, P.~Jimen{\'e}z-Seral, and V.~P{\'e}rez-Calabuig.
\newblock Soluble skew left braces and soluble solutions of the {Y}ang-{B}axter equation.
\newblock {\em Adv. Math.}, 455:109880, 2024.

\bibitem{categ_trip}
A.~Ballester-Bolinches, R.~Esteban-Romero, P.~P{\'e}rez-Altarriba, and V.~P{\'e}rez-Calabuig.
\newblock Categories of Skew Left Braces and Trifactorised Groups. 
\newblock {\em Commun. Math. Stat.}, doi:10.1007/s40304-025-00465-2, 2026.

\bibitem{Huppert}
B.~Huppert.
\newblock {\em Endliche Gruppen I}.
\newblock {Springer-Verlag, Berlin}, 1967.

\bibitem{truman}
P.~J.~Truman.
\newblock Analogues of Sylow's first theorem, Cauchy's theorem and Hall's theorem for skew braces.
\newblock {\em arXiv:2606.18414}, doi:10.48550/arXiv.2606.18414.
\end{thebibliography}

\end{document}